\documentclass[12pt]{article}

\usepackage{amsmath, amsthm, amssymb}
\usepackage{multirow}
\usepackage{multicol}
\usepackage{fullpage}
\usepackage{booktabs}
\usepackage{colortbl}
\usepackage{graphicx}
\usepackage{array}
\usepackage{setspace}

\theoremstyle{plain}
\newtheorem{theorem}{Theorem}
\newtheorem{lemma}[theorem]{Lemma}

\theoremstyle{definition}

\newcommand{\define}[1]{\emph{#1}}

\newcommand{\bitem}{\begin{itemize}}
\newcommand{\eitem}{\end{itemize}}
\newcommand{\benum}{\begin{enumerate}}
\newcommand{\eenum}{\end{enumerate}}

\newcommand{\ZZ}{\mathbb{Z}}

{\clearpage\newpage\section*{Corrigendum -- submitted #1}}{}

\ifx\volno\undefined\def\volno{16}\fi
\ifx\volyear\undefined\def\volyear{2009}\fi
\ifx\papno\undefined\def\papno{R00}\fi

\newfont{\footsc}{cmcsc10 at 8truept}
\newfont{\footbf}{cmbx10 at 8truept}
\newfont{\footrm}{cmr10 at 10truept}

\begin{document}

\title{A Generalization of Plexes of Latin Squares}
\author{Kyle Pula \\
\small Department of Mathematics\\[-0.8ex]
\small University of Denver, Denver, CO, USA\\
\small \texttt{jpula@math.du.edu}
}

\maketitle

\begin{abstract}
A \define{$k$-plex} of a latin square is a collection of cells representing each row, column, and symbol precisely $k$ times. The classic case of $k=1$ is more commonly known as a \define{transversal}. We introduce the concept of a \define{$k$-weight}, an integral weight function on the cells of a latin square whose row, column, and symbol sums are all $k$. We then show that several non-existence results about $k$-plexes can been seen as more general facts about $k$-weights and that the weight-analogues of several well-known existence conjectures for plexes actually hold for $k$-weights.
\end{abstract}

{\em Keywords:} latin square, transversal, duplex, k-plex.

{\em 2000 AMS Subject Classification:} 05B15.

\section{Introduction \& Background} \label{SEC:Introduction}

We call an integral weight function on the cells of a latin square a \define{$k$-weight} if the sum over each row, column, and symbol is $k$. The primary purpose of this paper is to show that several important non-existence results and existence conjectures about $k$-plexes hold in the much weaker setting of $k$-weights. Our hope, therefore, is that $k$-weights may prove to be a useful generalization to better understand $k$-plexes.

In \S \ref{SEC:DeltaLemma} we establish a simple lemma that is employed in several of our arguments. In \S \ref{SEC:ExistNonExist} we generalize a non-existence result of Wanless for odd-plexes to show that certain latin squares have no odd-weights. We also give a construction to show that analogues of conjectures of Ryser and Rodney about $k$-plexes hold for $k$-weights. In \S \ref{SEC:NearPlexes} we generalize recent results of Stein and Szab\'o concerning near transversals in Abelian groups to analogous objects related to $k$-weights. We close with \S \ref{SEC:OpenProblems} in which we mention several open questions about $k$-plexes and $k$-weights.

A \define{latin square} of order $n$ is an $n \times n$ array each of whose cells contains a symbol from a fixed $n$-set such that no symbol appears twice in any row or column. We can equivalently think of a latin square $L$ as set of triples with $(x,y,z) \in L \subset [n]^3$ if and only if the cell in row $x$ and column $y$ contains symbol $z$. Cayley tables of finite groups and other algebraically interesting loops provide among the most convenient and structured examples of latin squares. The standard references on latin squares are due to D\'enes and Keedwell \cite{Denes:1974ul,Denes:1991lq} and Laywine and Mullen \cite{MR1644242}.

For latin squares $L$ and $L'$, we say that \define{$L$ has the block pattern of $L'$} if $L$ can be represented by a block matrix $[A_{i,j}]_{1 \leq i, j \leq n}$ where each $A_{i,j}$ is itself a latin square and blocks $A_{i,j}$ and $A_{i',j'}$ contain the same symbols if and only if $L' (i,j) = L' (i',j')$. In such a case, it follows that each block has the same size, say $q$, and we say that $L$ has the \define{$q$-block pattern of $L'$}. When we are concerned only with the parity of $q$, we also refer to the odd-block or even-block pattern. In more classical terminology, a latin square with the $q$-block pattern of $(\ZZ_m, +)$ is said to be of \define{$q$-step type} and order $qm$.

A \define{$k$-plex} is a collection of cells of a latin square that meets each row, column, and symbol precisely $k$ times. The cases $k=1$ and $k=2$ correspond to \define{transversals} and \define{duplexes}, respectively. The study of transversals dates back at least to Euler \cite{Euler:Transversals} and is motivated in part by their intimate connection with mutually orthogonal latin squares. Wanless provides a helpful historical background on the more recent study of $k$-plexes for $k \ge 3$ \cite{Wanless:2002rr}.

Generalizing transversals in another direction, a \define{partial transversal} of length $t$ is a subset of $t$ cells, no two from the same row, column, or symbol, and such a subset is said to be \define{maximal} if it is not properly contained in another partial transversal. When $t$ is one less than the order of the latin square, we call the subset a \define{near transversal}.

Observe that a $k$-plex of a latin square $L$ is equivalent to a map from $L$ to the set $\{0,1\}$ such that the row, column, and symbol sums are all $k$. From this perspective, it is natural to explore the case when the set $\{0,1\}$ is replaced by some other set, say $\ZZ$. It is precisely this case that we consider in this paper. Put another way, the existence of a $k$-plex is equivalent to the satisfaction of a certain $(0,1)$-programming problem, and in this paper we consider the analogous problem over $\ZZ$.

Suppose $\theta: L \rightarrow \ZZ$ is an integral weight function on the cells of $L$. For $k \in \ZZ$, we call $\theta$ a \define{$k$-weight of $L$} if its sum over each row, column, and symbol is $k$. That is, for each index $i$, we have
\[
\sum_{(r,c,i) \in L} \theta(r,c,i) =
\sum_{(r,i,s) \in L} \theta(r,i,s) =
\sum_{(i,c,s) \in L} \theta(i,c,s) = k.
\]
We call $\theta$ a \define{partial $k$-weight} of $L$ with length $t$ if precisely $t$ row, $t$ column, and $t$ symbol sums are $k$ with each remaining sum being $0$. We say that $\theta$ \define{misses} those rows, columns, and symbols whose sums are $0$. When $t$ is one less than the order of $L$, we call $\theta$ a \define{near $k$-weight} of $L$. A partial $k$-weight is said to be $\define{maximal}$ if, as a vector in $\ZZ^{n^2}$, it is not dominated by another partial $k$-weight.

The reader may recognize that partial $k$-weights are closely analogous to $k$-homogenous partial latin squares.

\section{Lemma on Partial $k$-weights of Abelian Groups}\label{SEC:DeltaLemma}

We first recall a lemma of Paige that inevitably comes up in this context.

\begin{lemma}[Paige \cite{Paige:1947jt}]\label{LEM:Paige-Involution}
Suppose $(G, +)$ is a finite Abelian group. If $G$ has a unique involution, then it is equal to $\sum_{g \in G} g$. Otherwise, $\sum_{g\in G} g = 0$.
\end{lemma}

The following lemma plays a central role in several of our arguments. It is essentially a simplified version of an argument used by Egan and Wanless to show, among many other things, that latin squares with the odd-block pattern of $(\ZZ_{2m}, +)$ do not contain odd-plexes \cite{Egan:2008eu, Wanless:2002rr}. Our contribution has been to show that the argument applies more generally to partial $k$-weights. 

\begin{lemma}\label{LEM:Basic-Sum-Argument}
Suppose $L$ is the Cayley table of an Abelian group $(G, +)$ and $\theta$ is a partial $k$-weight whose missing rows $R$ sum to $r$, missing columns $C$ sum to $c$, and missing symbols $S$ sum to $s$. Then
\[
k(s - r - c) = 
\begin{cases}
\sum_{g \in G} g & \text{if $k$ is odd and $G$ has a unique involution} \\
0 & otherwise.
\end{cases}
\]
Note that when $\theta$ is a $k$-weight (rather than just a partial $k$-weight), $r=c=s=0$ and thus the left-hand side is always $0$.
\end{lemma}

\begin{proof}
Set $u := \sum_{g \in G} g$. First we consider the sum
\begin{align*}
\sum_{(x,y,z) \in L} \theta(x,y,z) (z - x - y)
&= \sum_{(x,y,z) \in L} \theta(x,y,z) z
- \sum_{(x,y,z) \in L} \theta(x,y,z) x
- \sum_{(x,y,z) \in L} \theta(x,y,z) y.
\end{align*}
We will evaluate the left-hand sum by examining each right-hand sum individually but first note that the result must be $0$ since $z - x - y = 0$ for every triple $(x,y,z) \in L$. Grouping the first of the three sums by the $z$ coordinate, we have
\begin{align*}
\sum_{(x,y,z) \in L} \theta(x,y,z) z
&= \sum_{z \in G} \left( \sum_{(x,y,z) \in L} \theta(x,y,z) \right) z \\
&= \sum_{z \in G \setminus S} k z \\
&= \sum_{z \in G} k z - \sum_{z \in S} k z \\
&= ku - k s
\end{align*}
Likewise, we have
\begin{align*}
\sum_{(x,y,z) \in L} \theta(x,y,z) x &= ku - k r \quad \text{and} \\
\sum_{(x,y,z) \in L} \theta(x,y,z) y &= ku - k c.
\end{align*}
Recalling that the original sum must be $0$, we now have
\begin{align*}
0 &= (ku - ks) - (ku - kr) - (ku - k c) \\
&= -ku + k(s - r - c).
\end{align*}
Thus $k(s-r-c) = ku$. By Lemma \ref{LEM:Paige-Involution}, if $k$ is even or $G$ does not have a unique involution, then $ku = 0$. Otherwise, $k$ is odd and $u$ is an involution. Therefore $ku = u$.
\end{proof}

\section{The Existence and Non-existence of $k$-weights}\label{SEC:ExistNonExist}

Ryser conjectured that every latin square of odd order has a transversal. His original, more ambitious conjecture was that the number of transversals of a latin square is equivalent to the square's order modulo $2$. Balasubramanian proved this form of the conjecture for even orders \cite{Balasubramanian:1990dk} but it fails for many latin squares of odd order beginning with $7$. As an aside, Akbari and Alipour have extended Balasubramanian's method to show that in any latin square the number of near transversals is equivalent to $0$ modulo $4$ \cite{MR2079255}.

Rodney \cite[p 105]{Colbourn:1996cj} conjectured that every latin square has a duplex. Both Ryser's and Rodney's conjectures have been shown by computer search to hold at least up to order $9$ by \cite{McKay:2006nx} and \cite{Egan:le}, respectively, and both are subsumed by a stronger conjecture again attributed to Rodney that every latin square can be partitioned into a collection of transversals and duplexes. For a more detailed discussion of the origins of these conjectures one should consult \cite{Egan:le}.

As an aside, a weaker but potentially more tractable question is whether every latin square can be partitioned into proper $k$-plexes, i.e. does every latin square of order $n$ have a $k$-plex for some $k < n$? The answer to even this seemingly weak question remains unknown.

Our first theorem shows that the natural analogues of Ryser's and Rodney's conjectures hold for $k$-weights.

\begin{theorem}\label{THM:k-weight-construction}
Every latin square has a $2$-weight and those of odd order have $1$-weights.
\end{theorem}

\begin{proof}
Fix any cell $(r,c,s) \in L$ and define $\theta: L \rightarrow \ZZ$ as follows:
\[
\theta(x,y,z) =
\begin{cases}
3-n & \text{if $(x,y,z) = (r,c,s)$}\\
1 & \text{if $(x,y,z)$ and $(r,c,s)$ agree in precisely one position}\\
0 & \text{otherwise}.
\end{cases}
\]
We claim that $\theta$ is a $2$-weight of $L$. First consider the sum of $\theta$ over row $x \neq r$. All cells in row $x$ have been assigned $0$ except for the cell in column $c$ and the cell containing symbol $s$. These exceptions must be distinct cells since $L$ is a latin square. Since these two cells both carry a weight of $1$, the sum over row $x$ is $2$. Now consider the sum over row $r$. Every cell in row $r$ carries weight $1$ except for cell $(r,c,s)$, which carries weight $3-n$. Thus the sum over row $r$ is $(n-1) + (3-n) = 2$.

Since our construction treats rows, columns, and symbols symmetrically, it follows that all column and symbol sums are also $2$.

Now suppose $n = 2m+1$ is odd. Note that every latin square has at least one $n$-weight since we may assign $1$ to every cell. Let $\theta$ and $\gamma$ be a $2$-weight and $n$-weight of $L$, respectively. Then $\psi := \gamma - m \theta$ is a $1$-weight of $L$.
\end{proof}

Of course it follows from the proof of Theorem \ref{THM:k-weight-construction} that the existence question for $k$-weights is rather crude. Each latin square has a $k$-weight either for all integers $k$ or for every even $k$. This fact contrasts sharply with the situation in $k$-plexes where the spectrum of existence can be much more subtle. Egan and Wanless, for example, have shown that for every even $n > 2$ there exists a latin square of order $n$ that has no $k$-plex for any odd $k < \lfloor n/4 \rfloor$ but has a $k$-plex for every other $k \le n/2$ \cite{Egan:2008eu}. For further results of this sort consult \cite{Egan:le,Bryant:2009zl,Egan:2010qa}.

Our next theorem has a long history. Euler showed that $(\ZZ_{2m}, +)$ has no transversals \cite{Euler:Transversals}, a century later Maillet showed the same result for any latin square with the odd-block pattern of $(\ZZ_{2m}, +)$ \cite{Maillet:qstep}, and another century passed before Wanless extended the result to odd-plexes \cite{Wanless:2002rr}. We show that the claim holds on the more general level of odd-weights.

\begin{theorem} \label{THM:NonExistence-odd-weights}
Let $L$ and $L'$ be latin squares.
\benum
\item The Cayley table of a finite Abelian group with a unique involution has no odd-weights.
\item If $L$ has no odd-weights and $L'$ has the odd-block pattern of $L$, then $L'$ has no odd-weights.
\eenum
\end{theorem}

The primary contribution of Theorem \ref{THM:NonExistence-odd-weights} is to show that the above sequence of results of Euler, Maillet, and Wanless follows from the more general setting of $k$-weights. However, we do show a bit more in that if there exists a latin square $L$ that has no odd-weights and does not have the odd-block pattern of $(\ZZ_{2m}, +)$, then by part (ii) of the theorem this property persists to all squares with the odd-block pattern of $L$. It remains an open question whether such squares exist.

\begin{proof}
(i) Suppose $\theta$ is a $k$-weight of $M$, the Cayley table of a finite Abelian group with unique involution $u$. By Lemma \ref{LEM:Basic-Sum-Argument}, $k$ is even since, otherwise, we immediately have the contradiction that $0 = u$.

(ii) Suppose $L'$ has the $q$-block pattern of $L$ and that $L$ has no odd-weights. Let $L'$ be represented by the block matrix $[A_{i,j}]_{1 \leq i,j \leq m}$ where squares $A_{i,j}$ and $A_{i',j'}$ use the same symbols if and only if $L(i,j) = L(i',j')$. Let $\theta$ be a $k$-weight of $L'$. Define the map $\psi : L \rightarrow \ZZ$ by
\[
\psi (i,j,k) := \sum_{(x,y,z) \in A_{i,j}} \theta(x,y,z).
\]

We now verify that $\psi$ is a $qk$-weight of $L$. Observe that the sum over row $r$ of $L$ equals the sum over $q$ different rows of $L'$. Since each row of $L'$ sums to $k$, the sum over row $r$ in $L$ equals $qk$. Similarly for columns and symbols. Thus $\psi$ is a $qk$-weight of $L$ but since $L$ has no odd-weight, either $q$ or $k$ must be even.
\end{proof}

\section{Near $k$-weights of Abelian Groups}\label{SEC:NearPlexes}

\begin{lemma}[Hall \cite{Hall:1952sf}]\label{LEM:Hall-NearTransversal}
The Cayley table of any finite Abelian group has a near transversal.
\end{lemma}

\begin{theorem}[Stein and Szab\'o \cite{Stein:2006rz}]\label{THM:Stein-Szabo}
Suppose $L$ is the Cayley table of an Abelian group of order $n$.
\benum
\item Then $L$ has a transversal or a maximal near transversal but not both.
\item If $n$ is prime, then there is no way to select a single cell from each row and column such that precisely two distinct symbols have been selected.
\eenum
\end{theorem}

As stated, Theorem \ref{THM:Stein-Szabo} is somewhat stronger than what \cite{Stein:2006rz} actually contains but our form follows easily from Lemma \ref{LEM:Hall-NearTransversal}, which Stein and Szab\'o use and discuss in their paper. We show that this result is again a more general fact about partial $k$-weights.

\begin{theorem}
Suppose $L$ is the Cayley table of an Abelian group of order $n$.
\benum
\item Then $L$ has a $1$-weight or a maximal near $1$-weight but not both.
\item If $\theta: L \rightarrow \ZZ$ whose row and column sums are all $1$, then it cannot happen that $n-2$ symbol sums are $0$ while the remaining sums are $n-i$ and $i$ with $\gcd(n,i) = 1$.\\
In particular, if $n$ is prime, then it cannot happen that precisely $n-2$ symbol sums are $0$.
\eenum
\end{theorem}

\begin{proof}
(i) Let $\theta$ be a near $1$-weight that misses row $r$, column $c$, and symbol $s$. We know such a partial weight exists by Lemma \ref{LEM:Hall-NearTransversal}. It suffices for our purposes to show that whether $\theta$ is maximal depends only on $G$ and not the particular partial weight.  Let $u := \sum_{g \in G} g$. By Lemma \ref{LEM:Basic-Sum-Argument}, 
\[
s - r - c = 
\begin{cases}
u & \text{if $G$ has a unique involution (which then must be $u$)} \\
0 & \text{otherwise}.
\end{cases}
\]
If $s - r - c = 0$, then $r + c = s$, i.e. $(r,c,s)$ is a cell in the Cayley table of $G$ and we may thus extend $\theta$ to a $1$-weight. If $s - r - c = u$, then $r + c \neq s$ and thus $\theta$ is maximal. In the case that $\theta$ is maximal, i.e. $G$ has a unique involution, we should also note that $G$ could not also have a $1$-weight since reducing the weight on any single cell would produce a near $1$-weight that by the preceding argument must be maximal, a contradiction.

(ii)
Suppose $\theta : L \rightarrow \ZZ$ has the property described in the statement of the theorem. In particular, all row and column sums are $1$ and all symbol sums are $0$ with the exception of symbols $g$ and $h$ whose sums are $i$ and $n-i$, respectively. As in the proof of Lemma \ref{LEM:Basic-Sum-Argument}, we evaluate the following trivial expression as three separate sums.
\begin{align*}
0 =\sum_{(x,y,z) \in L} \theta(x,y,z) (z - x - y)
&=  \sum_{(x,y,z) \in L} \theta(x,y,z) z - 2s \\
&= \sum_{(x,y,z) \in L} \theta(x,y,z) z \\
&= \sum_{z \in G} \bigg( \sum_{(x,y,z) \in L} \theta(x,y,z) \bigg) z \\
&= i g + (n-i) h \\
&= i(g - h). \\
\end{align*}
Thus either $g = h$ or the order of the nontrivial group element $g-h$ divides both $i$ and $n$. In the latter case, $\gcd(n,i) > 1$. 
\end{proof}

\section{Open Questions}\label{SEC:OpenProblems}
We introduced the concept of a $k$-weight of a latin square as a potentially useful generalization of a $k$-plex. We showed that several results about transversals and $k$-plexes can be seen as facts about these more general structures and that analogues of well-known conjectures about transversals and duplexes hold at least in the context of $k$-weights. We have yet to resolve at least two basic questions regarding $k$-weights.
\benum
\item We have seen that latin squares with the odd-block pattern of $\ZZ_{2m}$ have no odd-weights. Does the converse hold? That is, does the lack of odd-weights characterize latin squares of the odd-block pattern of $\ZZ_{2m}$? We suspect the answer is no but are not aware of any counter-example. The analogous question for plexes is also open: do there exist latin squares without odd-plexes besides those of odd-step type with an even number of blocks?

\item As we have seen, it was fairly straightforward to settle $k$-weight analogues of Ryser's and Rodney's conjectures by constructing $2$-weights for all latin squares and $1$-weights for the those of odd order. One might hope to find similar constructions for near $1$-weights in every latin square and thereby settle the $k$-weight analogue of Brualdi's conjecture that every latin square has a near transversal. It is interesting that this construction, should it exist, seems to be more difficult than the $1$-weight and $2$-weight constructions.
\eenum

\bibliographystyle{abbrv}

\begin{thebibliography}{10}

\bibitem{MR2079255}
S.~Akbari and A.~Alipour.
\newblock Transversals and multicolored matchings.
\newblock {\em J. Combin. Des.}, 12(5):325--332, 2004.

\bibitem{Balasubramanian:1990dk}
K.~Balasubramanian.
\newblock On transversals in {L}atin squares.
\newblock {\em Linear Algebra Appl.}, 131:125--129, 1990.

\bibitem{Bryant:2009zl}
D.~Bryant, J.~Egan, B.~Maenhaut, and I.~M. Wanless.
\newblock Indivisible plexes in latin squares.
\newblock {\em Designs, Codes and Cryptography}, 52(1), July 2009.

\bibitem{Colbourn:1996cj}
C.~J. Colbourn and J.~H. Dinitz.
\newblock {\em The CRC Handbook of Combinatorial Designs}.
\newblock CRC Press, Boca Raton, FL, 1996.

\bibitem{Denes:1974ul}
J.~D{\'e}nes and A.~D. Keedwell.
\newblock {\em Latin squares and their applications}.
\newblock Academic Press, New York, 1974.

\bibitem{Denes:1991lq}
J.~D{\'e}nes and A.~D. Keedwell.
\newblock {\em Latin squares}, volume~46 of {\em Annals of Discrete
  Mathematics}.
\newblock North-Holland Publishing Co., Amsterdam, 1991.

\bibitem{Egan:2010qa}
J.~Egan.
\newblock Bachelor latin squares with large indivisible plexes,
\newblock preprint.

\bibitem{Egan:le}
J.~Egan and I.~M. Wanless.
\newblock Indivisible partitions of latin squares.
\newblock {\em Journal of Statistical Planning and Inference}, to appear.

\bibitem{Egan:2008eu}
J.~Egan and I.~M. Wanless.
\newblock Latin squares with no small odd plexes.
\newblock {\em Journal of Combinatorial Designs}, 16(6):477--492, 2008.

\bibitem{Euler:Transversals}
L.~Euler.
\newblock Recherches sur une nouvelle esp{\`e}ces de quarr{\'e}s magique.
\newblock {\em Verh. uitgegeven door het Zeeuwsch Genootschap d. Wetensch. te
  Vlissingen}, pages 9:85--232, 1782.

\bibitem{Hall:1952sf}
M.~Hall, Jr.
\newblock A combinatorial problem on abelian groups.
\newblock {\em Proc. Amer. Math. Soc.}, 3:584--587, 1952.

\bibitem{MR1644242}
C.~F. Laywine and G.~L. Mullen.
\newblock {\em Discrete mathematics using {L}atin squares}.
\newblock Wiley-Interscience Series in Discrete Mathematics and Optimization.
  John Wiley \& Sons Inc., New York, 1998.

\bibitem{Maillet:qstep}
E.~Maillet.
\newblock Sur les carr\'es latins d'euler.
\newblock {\em C. R. Assoc. France Av. Sci}, 23(2):244--252, 1894.

\bibitem{McKay:2006nx}
B.~D. McKay, J.~C. McLeod, and I.~M. Wanless.
\newblock The number of transversals in a {L}atin square.
\newblock {\em Des. Codes Cryptogr.}, 40(3):269--284, 2006.

\bibitem{Paige:1947jt}
L.~J. Paige.
\newblock A note on finite {A}belian groups.
\newblock {\em Bull. Amer. Math. Soc.}, 53:590--593, 1947.

\bibitem{Stein:2006rz}
S.~K. Stein and S.~Szab{\'o}.
\newblock The number of distinct symbols in sections of rectangular arrays.
\newblock {\em Discrete Math.}, 306(2):254--261, 2006.

\bibitem{Wanless:2002rr}
I.~M. Wanless.
\newblock A generalisation of transversals for {L}atin squares.
\newblock {\em Electron. J. Combin.}, 9(1):Research Paper 12, 15 pp.
  (electronic), 2002.

\end{thebibliography}

\end{document}